\documentclass[11pt]{amsart}
\usepackage[english]{babel}
\usepackage{amssymb}
\usepackage{fourier}
\usepackage{mathrsfs}
\usepackage{enumerate}
\usepackage{color}
\usepackage{ifpdf}
\usepackage{tikz}
\usepackage{tikz-cd}
\usetikzlibrary{decorations.pathmorphing,arrows}
\usepackage[initials]{amsrefs}
\usepackage[all]{xy}

\usepackage{caption}
\usepackage{subcaption}
\usepackage{comment}
\usetikzlibrary{intersections}

\newcommand{\bbR}{{\mathbb R}}

\newcommand{\bfH}{{\mathbf H}}

\newcommand{\Isom}{\operatorname{Isom}}

\newcommand{\Ker}{\operatorname{Ker}}

\newcommand{\Stab}{\operatorname{Stab}}

\newcommand{\Aut}{\operatorname{Aut}}

\newcommand{\Homeo}{\operatorname{Homeo}}
\newcommand{\Commen}{\operatorname{Com}}
\newcommand{\ACommen}{\operatorname{AbCom}}
\newcommand{\GCommen}[2]{{\operatorname{GeCom}\left({#1},{#2}\right)}}

\newcommand{\BMSm}[1]{{m_{{#1}}^{\operatorname{BMS}}}}

\newcommand{\overto}[1]{{\buildrel{#1}\over\longrightarrow}}

\newcommand{\setdef}[2]{ \left\{ {#1}\ \left|\ {#2} \right.\right\} }

\newtheorem{theorem}{Theorem}[section]

\newtheorem{mthm}{Theorem}

\newtheorem{lemma}[theorem]{Lemma}

\newtheorem{corollary}[theorem]{Corollary}

\newtheorem{conjecture}[theorem]{Conjecture}

\theoremstyle{definition}

\newtheorem{defn}[theorem]{Definition}

\newtheorem{example}[theorem]{Example}

\newtheorem{remark}[theorem]{Remark}

\numberwithin{equation}{section}

\title[Arithmeticity and geometric commensurators]{Arithmeticity and geometric commensurators}

\author{Yanlong Hao}
\address{University of Michigan}
\email{yanlongqw@gmail.com}

\begin{document}
\date{\today}

\begin{abstract}
 We establish a coarse-geometric analogue of Margulis's commensurator criterion for groups acting on good Gromov-hyperbolic spaces. For a non-elementary $2\frac{1}{2}$-lattice $\Gamma$, we introduce its coarse-geometric commensurator and show that, if this commensurator contains $\Gamma$ with infinite index, then, after quotienting by a finite normal subgroup, either $\Gamma$ is a lattice in a nondiscrete totally disconnected locally compact group with dense commensurator, or $\Gamma$ is an arithmetic lattice in a rank-one simple real Lie group and its orbit metric is roughly isometric, up to scaling, to the corresponding arithmetic metric. We also prove an analogous characterization of rank-one locally symmetric metrics using boundary maps induced by rough isometries. In particular, assuming the Hilbert--Smith conjecture, the totally disconnected alternative is excluded when the limit set is a connected manifold.
\end{abstract}

\maketitle

\section{Introduction and Statement of the main results} 

Arithmetic locally symmetric spaces are central objects in geometric group theory and are closely connected to algebra, number theory, and combinatorics. Eberlein \cites{eberlein1980lattices, eberlein1982isometry} and Chen \cite{chen1980isometry} characterized symmetric spaces $H$ of noncompact type among all irreducible, simply-connected Riemannian manifolds with nonpositive sectional curvature in terms of their isometry groups. Following Farb and Weinberger \cite{farb2005hidden}, Van Limbeek \cite{van2014riemannian} extended this work by classifying all closed manifolds $(M,g)$ with a nondiscrete isometry group $\Isom(\tilde{M},g)$. Throughout the process, hidden symmetries \cite{farb2008isometries} were used to detect all closed, locally symmetric arithmetic manifolds. 

On the other hand, a well-known result by Margulis states that a lattice $\Gamma$ in a semisimple Lie group $G$ is arithmetic if and only if its commensurator is dense in $G$.

This paper aims to provide necessary and sufficient conditions for the arithmeticity of discrete subgroups acting isometrically on hyperbolic spaces, thereby generalizing Margulis’s aforementioned result. 

A $\Gamma$-action on a metric space $(X,d)$ is called \textbf{metrically proper} if, for all $x\in X$ and a real number $R>0$, the set
$\setdef{\gamma\in\Gamma}{d(x,\gamma x)<R}$
is finite. 

Let $(X,d)$ be a good hyperbolic space, and let $\Isom(X)$ be the group of isometries of $X$ with the bounded-open topology. Let $\Gamma<\Isom(X)$ be a discrete subgroup with $0<\delta_\Gamma<\infty$ that is of divergence type. As defined in \cite{furman2002coarse}, there is a Bowen--Margulis--Sullivan-type measure $\BMSm{\Gamma}$ on $\Lambda_X\Gamma^{(2)}X\times \bbR$. For more details, refer to Section~\ref{BMS}. For the properties and definitions of good hyperbolic spaces, see Section~\ref{good hyperbolic} and \cite{nica2016strong}. Here, we define a subgroup of $\Isom(X)$ as discrete if it acts metrically properly on $X$.

Notably, there is no standard definition of volume in a coarse-geometric setting.  Therefore, we adopt the following definition:

\begin{defn}\label{def:lattice}
 $\Gamma$ is a \textbf{$2\frac{1}{2}$-lattice} if $\Gamma$ is non-elementary and \[0<\BMSm{\Gamma}(\Lambda_X\Gamma^{(2)}X\times \bbR/\Gamma)<\infty.\]
\end{defn}
\begin{example}
    Any word-hyperbolic group is a $2\frac{1}{2}$-lattice acting on the good hyperbolic space associated by Mineyev to a Cayley graph \cite{mineyev2007metric}.
\end{example}

 Now consider an isometric $\Gamma$-action on $X$. Let $p$ be a base point of $X$. Let $(\Gamma,d)$ denote the pseudometric space such that $d(\gamma_1,\gamma_2)=d(\gamma_1p,\gamma_2p)$ for all $\gamma_1$, $\gamma_2\in \Gamma$. Therefore, the problem of classifying isometric actions of $\Gamma$ up to rough isometry is equivalent to the problem of classifying left-invariant quasi-metrics on $\Gamma$. Thus, we identify $(\Gamma,d)$ isometrically with the orbit $\Gamma p\subset X$. Throughout the paper, whenever we discuss the coarse geometry of $\Gamma$, we use this identification and freely regard maps of $(\Gamma,d)$ as maps of the orbit $\Gamma p$. We also restrict our attention to proper actions. 
 
 We call a left-invariant metric $(\Gamma,d)$ \textbf{arithmetic} if there exists a symmetric space $H$ of noncompact type and an embedding $i$ of $\Gamma$ into $\Isom(\bfH)$ as an arithmetic lattice such that $(\Gamma,\lambda d)$ is roughly isometric to $(i(\Gamma)o,d_H)$ for some point $o\in\bfH$, where $\lambda>0$.
\medskip

Now we turn to the second aspect of this paper. The purely group-theoretic concept related to the commensurator of a subgroup, known as the 
\textbf{abstract commensurator} 
\[
    \ACommen(\Gamma)
\]
of a general countable group $\Gamma$, is defined as the set of  
equivalence classes $[\phi]$
of group isomorphisms $\phi:\Gamma'\overto{}\Gamma''$ between finite-index subgroups
$\Gamma',\Gamma''<\Gamma$. Two isomorphisms $\phi_i:\Gamma'_i\overto{}\Gamma_i''$
($i=1,2$), are considered equivalent $\phi_1\sim\phi_2$ if there exists a finite-index
subgroup $\Gamma'<\Gamma_1'\cap\Gamma_2'$ such that the restrictions of $\phi_1$ and $\phi_2$
to $\Gamma'$ agree. The composition of (appropriate restrictions of) two such isomorphisms
defines a group structure on $\ACommen(\Gamma)$.

\begin{defn}\label{D:GCom}
    Let $\Gamma$ be a countable group and $d$ be a metric on $\Gamma$ as previously defined.
    The \textbf{coarsely geometric commensurator group}, which is denoted by
    $\GCommen{\Gamma}{d}$, is a subgroup of the abstract commensurator group 
    $\ACommen(\Gamma)$.
    It consists of classes $[\phi]$ of isomorphisms $\phi:\Gamma'\overto{}\Gamma''$
   that satisfy the following condition:
    \[
    \sup_{\gamma_1,\gamma_2\in\Gamma'}|d(\phi(\gamma_1),\phi(\gamma_2))-d(\gamma_1,\gamma_2)|<\infty.
    \]
    This group is independent of the choice of base point. 
\end{defn}

The main results presented in this paper are as follows.
\begin{mthm}\label{T: arithmetic}
Let $(X,d)$ be a good hyperbolic space. Suppose $\Gamma<\Isom(X)$ is a non-elementary $2\frac{1}{2}$-lattice and it acts metrically properly on $X$ with a compact limit set $\Lambda_X\Gamma$. There exists a finite normal subgroup $N$ of $\Gamma$ such that, writing $\bar\Gamma=\Gamma/N$,
if \[
        \left[\GCommen{\Gamma}{d}:\Gamma\right]=\infty,
\]
then either:
\begin{enumerate}
    \item there exists a nondiscrete, totally disconnected locally compact group $L$ such that $\bar\Gamma<L$ is a lattice and $\operatorname{Com}_L(\bar\Gamma)$ is dense in $L$; or
    \item $\bar\Gamma$ is embedded as an arithmetic lattice in a rank-one simple real Lie group $G$, and $(\Gamma,d)$ is roughly isometric to the arithmetic metric on $\bar{\Gamma}$.
\end{enumerate}
\end{mthm}

\medskip

We also seek a coarse-geometric characterization of locally symmetric metrics.

Let $\mathrm{rIsom}(X)$ denote the set of \textbf{rough isometries} on $(X, d)$, defined by
\[
\mathrm{rIsom}(X):=\setdef{f:X\rightarrow X}{|f_*d-d|<\infty, \quad f(X)\text{ is coarsely dense}}.
\]
Here, a subset $A\subset X$ is called \textbf{coarsely dense} if there exists a constant $C>0$ such that for all $x\in X$, there exists $y\in A$ such that $d(x,y)\leq C$.
The \textbf{group of rough isometries} is denoted by
$\mathrm{RIsom}(X)=\mathrm{rIsom}(X)/\sim,$ where $f\sim g$ if $|f-g|<\infty$.

An element $f\in\mathrm{RIsom}(X)$ induces a homeomorphism $\partial f\in\Homeo(\partial X)$. For every subgroup $\Gamma<\Isom(X)$, we define \[\mathrm{H}_{\Gamma}(X)=\setdef{\partial f|_{\Lambda_X\Gamma}}{f\in  \mathrm{RIsom}(X), \partial f(\Lambda_X\Gamma)=\Lambda_X\Gamma,\partial f|_{\Lambda_X\Gamma}\in H},\]
\[
\partial \Gamma=\setdef{\partial \gamma|_{\Lambda_X\Gamma}}{\gamma\in \Gamma}.
\]
where $H$ is defined in equation~\eqref{eq: def of H}.
\begin{mthm}\label{T: locally}
Let $(X,d)$ be a good hyperbolic space. Suppose that $\Gamma<\Isom(X)$ is a $2\frac{1}{2}$-lattice and that $\Gamma$ acts metrically properly on $X$ with a compact limit set $\Lambda_X\Gamma$. Then there exists a finite normal subgroup $N$ of $\Gamma$ such that, writing $\bar\Gamma=\Gamma/N$, 
if the index 
        $\left[\mathrm{H}_{\Gamma}(X):\partial \Gamma\right]$ is infinite,
 then either:
\begin{enumerate}
    \item there exists a nondiscrete, totally disconnected locally compact group $L$ such that $\bar\Gamma<L$ is a lattice; or
    \item $\bar\Gamma$ is embedded as a lattice in a rank-one simple real Lie group $G$, and $(\Gamma,d)$ is roughly isometric to the metric on $\bar{\Gamma}$ induced by this embedding.
\end{enumerate}
\end{mthm}
\subsection{Hilbert-Smith conjecture}
While Hilbert’s fifth problem has been solved \cite{zippin1955topological}, the Hilbert-Smith conjecture remains unresolved in most cases.
\begin{conjecture}[Hilbert-Smith]
  If a locally compact group acts continuously and faithfully on a connected manifold, then it must be a Lie group.  
\end{conjecture}
\begin{corollary}
Assuming the Hilbert-Smith conjecture and the hypotheses of Theorems~\ref{T: arithmetic} and ~\ref{T: locally}. If, in addition, $\Lambda_X \Gamma$ is a connected manifold, case (1) in both Theorems~\ref{T: arithmetic} and ~\ref{T: locally} cannot occur.
\end{corollary}
The Hilbert-Smith conjecture has been established in a few instances. A concise proof exists for groups acting by Lipschitz maps on a Riemannian manifold \cite{repovs1997proof}. John Pardon solved the three-dimensional case \cite{pardon2013hilbert}.

Therefore, unconditionally, we obtain the following result:
\begin{corollary}
    In the statements of Theorems~\ref{T: arithmetic} and ~\ref{T: locally}, if $\Lambda_X\Gamma$ is a connected topological manifold of dimension at most 3, then case (1) cannot occur.
\end{corollary}

The remainder of this paper is organized as follows: In Section~\ref{Sec: Pre}, we recall some preliminaries. Section~\ref{lattice} shows that a $2\frac{1}{2}$-lattice is a lattice in a suitable locally compact group $H$. Section~\ref{Sec: sturcture} examines the structure of $H$. Section~\ref{sec: proofs} presents the proofs of Theorems~\ref{T: arithmetic} and~\ref{T: locally}. 

\textbf{AI usage:} The work is original to the author. AI has been use along the writing for 1. check the notations and constants used is consistent; 2. grammar checking. 3. Draft the abstract. 

\section{Preliminaries}\label{Sec: Pre}
\subsection{Standing assumptions and boundary geometry} Let $(X,d)$ be a Gromov-hyperbolic space with Gromov boundary $\partial X$. Then $\bar{X}=X\cup \partial X$ is a bordification of $X$. Fix a base point $o\in X$. The limit set of a subgroup $\Gamma<\Isom(X,d)$ is defined as 
\[
\Lambda_X\Gamma=\overline{\Gamma o}^{\bar{X}}\cap \partial X.
\]
The group $\Gamma$ is called non-elementary if $\Lambda_X\Gamma$ contains at least three points. 

\subsection{Good hyperbolic spaces}\label{good hyperbolic}
Gromov-hyperbolic spaces play a fundamental role in modern mathematics. However, certain applications require stronger geometric properties.  In \cite{nica2016strong}, the concept of strong hyperbolic spaces was introduced. For our purposes, the following definition suffices.
\begin{defn}
   A hyperbolic space $X$ is called \textbf{$\epsilon$-good}, where $\epsilon>0$, if the following
two properties hold for each base point $p\in X$:
\begin{itemize}
    \item The Gromov product $(\cdot,\cdot)_p$ on $X$ extends continuously to the bordification $X\cup \partial X$.
\item The function $d_\epsilon:=\exp(-\epsilon (\cdot,\cdot)_p)$ defines a metric on the boundary $\partial X$.
\end{itemize}
\end{defn}
\begin{example}
If $\Gamma$ is a hyperbolic group, then the Green metric defined by a symmetric
and finitely supported random walk on $\Gamma$ is both a good hyperbolic and a strong hyperbolic metric \cite{nica2016strong}. 
\end{example}
\begin{example}
Let $\Gamma$ be a Gromov-hyperbolic group. The Cayley graph of $\Gamma$ with a finite generating set is a Gromov-hyperbolic space. Mineyev \cite{mineyev2007metric} constructed a good hyperbolic space with a cocompact isometric $\Gamma$-action. 
\end{example}

Another useful tool for studying good hyperbolic spaces $(X, d)$ is the Busemann function.
For $\xi\in\partial X$, the Busemann function $\beta_\xi: X\times X\to \bbR$ is defined by
\[
\beta_\xi(x,y)=\lim_{z\to \xi}d(x,z)-d(y,z).
\]
The limit is independent of the sequence converging to $\xi$ since $X$ is a good hyperbolic space.

\subsection{Cross-ratio}
The classical cross-ratio on $S^1$ can be generalized to the boundary of good hyperbolic spaces \cites{otal1992geometrie, paulin1996groupe, hamenstadt1997cocycles}.
Cross-ratios provide a unified interpretation of many geometric structures,
making them valuable tools for studying various spaces with group actions.

Let $Y$ be a space, and let $Y^{(n)}$ be the set of
$n$–tuples of pairwise distinct points in $Y$. A map
$\mathbb{B}: Y^{(4)}\rightarrow \bbR$ is called a cross-ratio if  it satisfies the following conditions:
\begin{align*}
 &\mathbb{B}(x,y,z,w)=-\mathbb{B}(y,x,z,w);\\
 &\mathbb{B}(x,y,z,w)=\mathbb{B}(z,w,x,y);\\
 &\mathbb{B}(x,y,z,w)=\mathbb{B}(x,y,z,t)+\mathbb{B}(x,y,t,w);\\
 &\mathbb{B}(x,y,z,w)+\mathbb{B}(y,z,x,w)+\mathbb{B}(z,x,y,w)=0.
\end{align*}
For an $\epsilon$-good hyperbolic space $(X,d_X)$, there exists a canonical cross-ratio on the boundary $\partial X$ given by
\begin{equation}\label{cross-ratio}
 \mathbb{B}(\xi,\eta,\xi',\eta')=\ln\left(\frac{d_\epsilon(\xi,\xi')d_\epsilon(\eta,\eta')}{d_\epsilon(\xi,\eta')d_\epsilon(\eta,\xi')}\right).   
\end{equation}

Otal observed that formula~\eqref{cross-ratio} is equivalent to
\[
\mathbb{B}(\xi,\eta,\xi',\eta')=\lim_{(x,y,x',y')\to (\xi,\eta,\xi',\eta')}\frac{1}{2}[d_X(x,x')+d_X(y,y')-d_X(x,y')-d_X(y,x')].
\]

\subsection{M\"{o}bius homeomorphisms}
Let $(K,d)$ be a compact metric space without isolated points, and consider
M\"{o}bius self-homeomorphisms of $K$. The cross-ratio of a quadruple of distinct points in $K$ is defined by
\[
\mathbb{B}(x,y,x',y')=\ln \left(\frac{d(x,x')d(y,y')}{d(x,y')d(y,x')}\right).
\]
A homeomorphism of $K$ is called M\"{o}bius if it preserves the cross-ratio.

Sullivan \cite{sullivan1979density} showed that M\"{o}bius maps satisfy the \textbf{geometric mean-value property} \cite{nica2013proper}*{Lemma 4.1}, namely, 

\emph{
A self-homeomorphism $g$ of $K$ is M\"{o}bius if and only if there exists
a positive continuous function on $K$, denoted $g'$, such that, for all $x, y\in K$, 
\[
g'(x)g'(y)=\frac{d^2(gx,gy)}{d^2(x,y)}.
\] 
}

Nica \cite{nica2013proper} has shown that the metric derivative $g'$ is Lipschitz. Moreover, the group of M\"{o}bius maps is locally compact by \cite{nica2013proper}*{Proposition 5.6}.

\subsection{Bowen-Margulis-Sullivan measure}\label{BMS} Independently discovered by Bowen and Margulis and later reconstructed from the Patterson-Sullivan measure, the Bowen-Margulis-Sullivan measure (BMS measure) of a negatively curved closed manifold contains critical geometric information. It serves as a tool to study the growth of closed geodesics.

Furman \cite{furman2002coarse} extended the BMS measure in a coarse geometric setting.
Here, we adopt a similar strategy. 

Let $\Gamma$ be a group that acts metrically properly on a good hyperbolic space $(X, o)$, with limit set $\Lambda_X \Gamma$. We define the Poincar\'{e} series of $\Gamma$ to be 
\[
P(s)=\sum_{\gamma\in \Gamma}e^{-sd(o,\gamma o)}.
\]
Recall that the \textbf{critical exponent} of $\Gamma$ is given by 
\[
\delta_\Gamma=\inf\{s| P(s)<\infty\},
\]
and $\Gamma$ is called of \textbf{divergence type} if $P(\delta_\Gamma)=\infty$. 
Assume that $\Gamma$ is of divergence type, that $0<\delta_\Gamma<\infty$, and that $\Lambda_X\Gamma$ is compact. Let $v$ denote the Patterson–Sullivan measure on $\Lambda_X\Gamma$. For the construction of the Patterson-Sullivan measure, see \cite{das2017geometry}*{Theorem 15.4.9}. 

There is a $\Gamma\times \bbR$ action on $\Lambda_X\Gamma^{(2)}\times \bbR$
\[
(\gamma,t)\cdot(\xi,\eta,s)=(\gamma \xi,\gamma \eta, s+\frac{\beta_\xi(p, \gamma p)-\beta_\eta(p,\gamma p)}{2}+t),
\]
where $\beta$ is the Busemann function with base point $p\in X$. A straightforward computation shows that the measure
$\BMSm{\Gamma}=d_{\epsilon}^{-\frac{2\delta_\Gamma}{\epsilon}}v\times v\times dt$ is $\Gamma\times \bbR$-invariant. This is the measure used in Definition~\ref{def:lattice}.

Notably, the assumption of divergence type is likely unnecessary. However, throughout the paper we assume that all groups under consideration are of divergence type.

\section{A $2\frac{1}{2}$-lattice is a lattice}\label{lattice}
Consider a good hyperbolic space $(X,d)$ and let $\Gamma<\Isom(X)$ be a $2\frac{1}{2}$-lattice with a compact limit set. In this section, we establish that $\Gamma$ is virtually a lattice in a locally compact group $H$. The properties of $H$ are further analyzed in the next section.

\subsection{M\"{o}bius homomorphism and $H$}
 The cross-ratio $\mathbb B$ on $\partial X$, when restricted to the limit set $\Lambda_X \Gamma$, is $\Gamma$-invariant. Let 
\[
\operatorname{Mob}:=\setdef{f\in \Homeo(\Lambda_X \Gamma)}{f^{(4)}\mathbb B=\mathbb B}.
\]
$\operatorname{Mob}$ is a locally compact group in the compact-open topology.

With respect to the metric $d_\epsilon$, each $f\in \operatorname{Mob}$ is an M\"{o}bius homeomorphism.  Importantly, when $f$ is the boundary map of an isometry $F$, the metric derivative is related to the Busemann function as follows:
\[
f'(\xi)=\exp(\epsilon \beta_\xi(p,F^{-1}p)).
\]

We now define the group $H$, which depends on the measure $\BMSm{\Gamma}$.

Let $v$ be the Patterson-Sullivan measure on $\Lambda_X\Gamma$. We define $H$ by 
\begin{equation}\label{eq: def of H}
H:=\setdef{f\in\Homeo(\Lambda_X \Gamma)}{(f\times f)_*(d_\epsilon^{-\frac{2\delta_\Gamma}{\epsilon}}v\times v)=d_\epsilon^{-\frac{2\delta_\Gamma}{\epsilon}}v\times v}.
\end{equation}
By the definition of the cross-ratio, all $f\in H$ preserve the cross-ratio almost everywhere with respect to the measure $v^{4}$. Since the cross-ratio is continuous, it follows that $f\in \operatorname{Mob}$. Therefore, $H$ is a closed subgroup of $\operatorname{Mob}$ and hence is locally compact.

\subsection{$\partial \Gamma$ is discrete}
To prove that $\partial \Gamma$ is discrete in $H$, we combine the following three lemmas. Note that it is unclear whether $H$ is a subgroup of $\Homeo(\partial X)$; therefore, the discreteness of $\Gamma$ as a subgroup of $\Isom(X)$ is insufficient. 

Lemmas 3.1 and 3.2 are standard for proper spaces. Since we are not restricted to locally compact spaces, we provide short proofs here.

Let $\partial:\Isom(X,d)\rightarrow \Homeo(\partial X)$ be the boundary map.
\begin{lemma}\label{lem: ker is finite}
Suppose that $\partial X$ contains at least three points.
Then the kernel $F=\Ker(\partial)$ has a bounded orbit. 
\end{lemma}
\begin{proof}
Fix $f\in F$ and let $x=f(p)$. We first observe that the metric derivative of $f$ is identically equal to $1$. Indeed, since $f$ induces the identity map on $\partial X$, the geometric mean-value property gives
\[
f'(\xi)f'(\eta)=1
\]
for every pair of distinct points $\xi,\eta\in\partial X$. Since $\partial X$ contains at least three points and $f'$ is positive, it follows that
\[
f'(\xi)=1
\]
for every $\xi\in\partial X$. Hence, using the relation between the metric derivative and the Busemann function, we have
\[
\beta_\xi(x,p)=0
\]
for every $\xi\in\partial X$.

Let $w_n\in X$ converge to $\xi\in\partial X$. Then
\[
\begin{aligned}
2(x,\xi)_p
&=2\lim_{n\to\infty}(x,w_n)_p=\lim_{n\to\infty}
\bigl(d(x,p)+d(w_n,p)-d(x,w_n)\bigr)\\
&=d(x,p)-\beta_\xi(x,p)=d(x,p).
\end{aligned}
\]
Thus,
\[
(x,\xi)_p=\frac12 d(x,p)
\]
for every $\xi\in\partial X$.

Fix two distinct points $\xi,\eta\in\partial X$. By the hyperbolicity inequality,
\[
(\xi,\eta)_p
\geq
\min\{(\xi,x)_p,(x,\eta)_p\}-\delta
=\frac12d(x,p)-\delta.
\]
Therefore,
\[
d(x,p)\leq 2(\xi,\eta)_p+2\delta.
\]
The right-hand side is independent of $f$. Hence the orbit $F\cdot p$ is bounded.
\end{proof}

Since $\Gamma$ acts metrically properly on $X$, the kernel of the boundary map $\partial:\Gamma\to \partial \Gamma$ is finite. The discreteness of $\partial \Gamma$ follows from a similar argument.
\begin{lemma}\label{L: finite}
    For any $K>0$, the set $\Omega(K):=\setdef{\gamma\in \partial \Gamma}{\sup_{\xi\in \partial X}|\frac{1}{\epsilon}\log(f^{-1})'(\xi)|\leq K}$ is finite. In particular,
    $\partial \Gamma$ is discrete in $\Homeo(\partial X)$. 
\end{lemma}
\begin{proof}
Let $f\in\Omega(K)$. Then, for every $\xi\in\partial X$,
\[
\left|\frac{1}{\epsilon}\ln\bigl((f^{-1})'(\xi)\bigr)\right|\leq K.
\]

Fix a lift of $f$ to $\Gamma$, which we again denote by $f$, and let $y=f(p)$.
As in the proof of the previous lemma, for every $\xi\in\partial X$,
\[
2(\xi,y)_p\geq d(y,p)-K.
\]

Fix two distinct points $\xi,\eta\in\partial X$. By the hyperbolicity inequality,
\[
(\xi,\eta)_p
\geq
\min\{(\xi,y)_p,(y,\eta)_p\}-\delta.
\]
Therefore,
\[
2(\xi,\eta)_p\geq d(y,p)-K-2\delta,
\]
and hence
\[
d(y,p)\leq 2(\xi,\eta)_p+K+2\delta.
\]

The right-hand side is independent of $f$. Thus, every element of $\Omega(K)$ admits a lift in $\Gamma$ whose displacement of $p$ is uniformly bounded. Since the action of $\Gamma$ on $X$ is metrically proper, there are only finitely many such lifts. Consequently, $\Omega(K)$ is finite.
\end{proof}

By restricting the action of $\partial \Gamma$ on $\partial X$ to $\Lambda_X\Gamma$, we obtain a map $\psi: \partial \Gamma\to H$.
\begin{lemma}\label{L: finite in H}
 $\Ker(\psi)$ is finite and $\psi(\partial \Gamma)$ is discrete in $H$.
\end{lemma}
\begin{proof}
 Assume that there is a sequence of elements $\{\gamma_n\}$ such that $\psi(\gamma_n)$ converges to the identity map. After taking a subsequence, we may assume that for all $\xi\in \Lambda_X\Gamma$,
 $\frac{1}{2}\leq(\gamma_n)'(\xi)\leq 2.$
 
Let $\eta\neq\zeta\in \Lambda_X\Gamma$. By the geometric mean-value property, we have
\[
\frac{1}{4}\leq\frac{d_\epsilon(\gamma_n\eta,\gamma_n\zeta)}{d_\epsilon(\eta,\zeta)}\leq 4.
\]
Consider any point $\theta\in\partial X$. By the triangle inequality and the definition of the visual metric, after interchanging $\eta$ and $\zeta$ if necessary, we may assume that
\[\frac{d_\epsilon(\eta,\zeta)}{8}\leq\frac{d_\epsilon(\gamma_n\eta, \gamma_n\zeta)}{2}\leq d_\epsilon(\gamma_n\theta,\gamma_n\eta)\leq 1.\]

Using the geometric mean-value property, we obtain 
\[
\gamma_n'(\theta)=\frac{1}{\gamma_n'(\eta)}\frac{d_\epsilon^2(\gamma_n\theta,\gamma_n\eta)}{d_\epsilon^2(\theta,\eta)}\geq \frac{1}{2}\frac{d_\epsilon^2(\eta,\zeta)}{64}.
\]
Similarly, 
\[
\gamma_n'(\theta)\leq \frac{2}{d_\epsilon^2(\eta,\zeta)}.
\]
By Lemma~\ref{L: finite}, the sequence $\{\gamma_n\}$ contains only finitely many distinct elements and becomes eventually constant.
\end{proof}

From now on, by abuse of notation, we identify the image $\psi(\partial\Gamma)$ with $\partial\Gamma$.

\begin{lemma}\label{precompact}
Let $A$ be a subset of $H$. Suppose that there exist constants
$\kappa,M>0$ such that, for every $f\in A$, there are distinct points
$\xi_f,\eta_f\in\Lambda_X\Gamma$ satisfying
\begin{enumerate}
\item $d_\epsilon(\xi_f,\eta_f)\geq\kappa$ and $d_\epsilon(f\xi_f,f\eta_f)\geq\kappa$;
\item $e^{-M}\leq f'(\xi_f)\leq e^M$.
\end{enumerate}
Then $A$ is precompact in $H$.
\end{lemma}
\begin{proof}
The argument is the same as in Lemma~\ref{L: finite in H}. By the geometric mean-value property,
\[
f'(\xi_f)f'(\eta_f)=\frac{d^2_\epsilon(f\xi_f,f\eta_f)}{d^2_\epsilon(\xi_f,\eta_f)}.
\]
The separation assumptions and the bound on $f'(\xi_f)$ therefore give uniform upper and lower bounds on $f'(\eta_f)$. The argument in the
proof of Lemma~\ref{L: finite in H} then gives uniform upper and lower bounds for $f'$ on all of $\Lambda_X\Gamma$. Hence $A$ is precompact.
\end{proof}

\subsection{$\partial \Gamma$ is a lattice}
We are now ready to prove the main theorem of this section.

We start with a few lemmas.
\begin{lemma}
The action of $H$ on
\[
\mathcal{G}:=(\Lambda_X\Gamma)^{(2)}\times\mathbb{R}
\]
defined by
\[
h(\xi,\eta,t)=\left(h\xi,h\eta,t+\frac{\log h'(\xi)-\log h'(\eta)}{2\delta_\Gamma}\right)
\]
is proper.
\end{lemma}

\begin{proof}
Let $C_1,C_2\subset\mathcal{G}$ be compact, and set
\[
B=\{h\in H:hC_1\cap C_2\neq\emptyset\}.
\]
For every $h\in B$, choose $(\xi_h,\eta_h,t_h)\in C_1$ such that
\[
h(\xi_h,\eta_h,t_h)\in C_2.
\]

Since $C_1$ and $C_2$ are compact subsets of $(\Lambda_X\Gamma)^{(2)}\times\mathbb{R}$, there exist constants $a,b>0$ and $R>0$ such that
\[
d_\epsilon(\xi_h,\eta_h)\geq a,\qquad d_\epsilon(h\xi_h,h\eta_h)\geq b,
\]
and
\[\left|\log h'(\xi_h)-\log h'(\eta_h)\right|\leq R
\]
for every $h\in B$.

By the geometric mean-value property,
\[
h'(\xi_h)h'(\eta_h)=\frac{d^2_\epsilon(h\xi_h,h\eta_h)}{d^2_\epsilon(\xi_h,\eta_h)}.
\]
Hence both
\[
\log h'(\xi_h)+\log h'(\eta_h)
\quad\text{and}\quad
\log h'(\xi_h)-\log h'(\eta_h)
\]
are uniformly bounded. It follows that $h'(\xi_h)$ and
$h'(\eta_h)$ are uniformly bounded above and away from zero.

By Lemma~\ref{precompact}, $B$ is precompact in
$\operatorname{Mob}(\Lambda_X\Gamma)$.

Finally, $B$ is closed because the action is continuous and $C_1$
and $C_2$ are compact. Thus, $B$ is compact, and the action is
proper.
\end{proof}

\begin{theorem}\label{THM: Lattice of H}
The image $\psi(\partial \Gamma)$ is a lattice in $H$.
\end{theorem}
\begin{remark}
    The condition that a group is a lattice in a locally compact, second-countable group is very general. Indeed, any countable group is a lattice in itself or any finite extension of itself. Moreover, in most cases, the group $H$ constructed here is discrete.
\end{remark}
\begin{proof}
In this proof, for simplicity, we assume $\Gamma=\psi(\partial \Gamma)$. Let
\[
Z=(\Lambda_X\Gamma)^{(2)}\times\mathbb{R}.\]
 
 Following \cite{furman2002coarse}, consider the abstract geodesic flow 
 \[H\rightarrow \Homeo(Z),\quad
f(x,y,t)=(f(x),f(y),t+\frac{\log f'(x)-\log f'(y)}{2\delta_\Gamma}).
\]

First, the quotient space $Y=H\backslash Z$ is locally compact and second countable. The space $Z$ is locally compact and second-countable. The second countability is clear. Since the $H$-action is continuous and proper, the quotient $Y$ is Hausdorff and the quotient map $q:Z\to Y$ is open. Given $y=q(x)$, choose a relatively compact open neighborhood $U$ of $x$. Then $q(U)$ is an open neighborhood of $y$, and $\overline{q(U)}\subseteq q(\overline{U})$
which is compact. Hence $Y$ is locally compact.

Let
\[
m=\BMSm{\Gamma}
\]
denote the Bowen--Margulis--Sullivan measure on $Z.$
By the definition of $H$, the measure $m$ is $H$-invariant. Moreover, the measure induced by $m$ on $\overline Y=\Gamma\backslash Z$
is finite.

Let
\[
p:\overline Y\longrightarrow Y
\]
be the natural projection. By assumption, the Bowen--Margulis--Sullivan measure induces a finite Radon measure, which we continue to denote by $\BMSm{\Gamma}$, on $\overline Y$. Set
\[
\nu=p_*\BMSm{\Gamma}.
\]
By the disintegration theorem, there exists a $\nu$-measurable family of probability measures
\[
\{\mu_y:y\in Y\}
\]
such that $\mu_y$ is supported on $p^{-1}(y)$ for $\nu$-almost every $y$ and
\[
\BMSm{\Gamma}
=
\int_Y\mu_y\,d\nu(y).
\]

Fix $y\in Y$ and choose $z\in Z$ whose image in $Y$ is $y$. Let
\[
K_z=\Stab_H(z).
\]
Since the action of $H$ on $Z$ is proper, $K_z$ is compact, and the orbit map identifies
\[
q^{-1}(y)=H\cdot z\cong H/K_z,
\]
where
\[
q:Z\longrightarrow Y
\]
is the quotient map. Moreover,
\[
p^{-1}(y)
\cong
\Gamma\backslash H/K_z.
\]

Lift the conditional measure $\mu_y$ from $\Gamma\backslash H/K_z$ to a $\Gamma$-invariant Radon measure $\widetilde\mu_y$ on $H/K_z$. Explicitly, for every compactly supported Borel function $\varphi$ on $H/K_z$, set
\[
\int_{H/K_z}\varphi\,d\widetilde\mu_y
=
\int_{\Gamma\backslash H/K_z}
\left(
\sum_{\gamma\in\Gamma/\Gamma_x}
\varphi(\gamma x)
\right)
d\mu_y([x]),
\]
where
\[
\Gamma_x=\Stab_\Gamma(x).
\]
Since the $\Gamma$-action is proper, the sum is locally finite, and hence $\widetilde\mu_y$ is a Radon measure.

By construction, the family
\[
\{\widetilde\mu_y:y\in Y\}
\]
is a disintegration of the Bowen--Margulis--Sullivan measure on $Z$ over the map
\[
q:Z\longrightarrow Y.
\]
More precisely, for every compactly supported Borel function $\varphi$ on $Z$,
\[
\int_Z\varphi\,d\BMSm{\Gamma}
=
\int_Y
\left(
\int_{q^{-1}(y)}
\varphi\,d\widetilde\mu_y
\right)
d\nu(y).
\]

We now use the $H$-invariance of the Bowen--Margulis--Sullivan measure on $Z$. Let $h\in H$. Since
\[
q\circ h=q
\]
and
\[
h_*\BMSm{\Gamma}=\BMSm{\Gamma},
\]
the family
\[
\{h_*\widetilde\mu_y:y\in Y\}
\]
is another disintegration of the same measure over $q$ with respect to the same measure $\nu$. By uniqueness of disintegration,
\[
h_*\widetilde\mu_y=\widetilde\mu_y
\]
for $\nu$-almost every $y$.

Since $H$ is second countable, we may first apply this argument to a countable dense subgroup of $H$ and then use continuity of the action on Radon measures. Thus, after discarding a $\nu$-null subset, $\widetilde\mu_y$ is $H$-invariant for $\nu$-almost every $y$.

Under the identification
\[
q^{-1}(y)\cong H/K_z,
\]
the measure $\widetilde\mu_y$ is therefore an $H$-invariant Radon measure on $H/K_z$. Since $K_z$ is compact, such a measure is unique up to multiplication by a positive scalar. Hence
\[
\widetilde\mu_y
=
c_y\,m_{H/K_z}
\]
for some $c_y>0$.

The measure induced by $\widetilde\mu_y$ on
\[
\Gamma\backslash H/K_z
\]
is precisely $\mu_y$. Since $\mu_y$ is a probability measure, it follows that
\[
m_{H/K_z}\bigl(\Gamma\backslash H/K_z\bigr)<\infty.
\]
As $K_z$ is compact, this implies
\[
m_H(\Gamma\backslash H)<\infty.
\]
Therefore $\Gamma$ is a lattice in $H$.
\end{proof}
\begin{corollary}\label{C: intermediate lattice}
Let $L<H$ be a closed subgroup containing $\partial\Gamma$. Then $\partial\Gamma$ is a lattice in $L$.
\end{corollary}
\begin{proof}
The proof of Theorem~\ref{THM: Lattice of H} applies verbatim with $H$ replaced by $L$. Indeed, the action of $L$ on
\[
(\Lambda_X\Gamma)^{(2)}\times\mathbb R
\]
is proper, and $L$ preserves the Bowen--Margulis--Sullivan measure. Therefore,
\[
m_L(\partial\Gamma\backslash L)<\infty,
\]
and hence $\partial\Gamma$ is a lattice in $L$.
\end{proof}

\section{Structure of $H$}\label{Sec: sturcture}
The structure of the group $H$ can be determined by studying the properties of the group $\Gamma$. 

The \textbf{amenable radical} of a locally compact second-countable group $G$, denoted by $\mathrm{Rad_{am}}(G)$, is the largest closed normal amenable subgroup of $G$. The starting point for the structure of $H$ is a consequence of Hilbert’s fifth problem, which was first observed by Burger and Monod \cite{burger2001continuous}.
\begin{theorem}\cite{burger2001continuous}*{Theorem 3.3.3}\label{T: Stucture}
Every locally compact group $G$ contains an open, normal, finite-index subgroup $G_0$
that contains $\mathrm{Rad_{am}}(G)$. The quotient $G_0/ \mathrm{Rad_{am}}(G)$ is topologically isomorphic to a direct product of a connected, center-free, semisimple, real Lie group without
compact factors and a totally disconnected locally compact group with trivial amenable radical.
\end{theorem}

To understand the structure of $H$, we first focus on its amenable radical, which is related to the (CAF) condition of $\Gamma$. The (CAF) condition states that every amenable commensurated subgroup of $\Gamma$ is finite.

We use barycenter methods to show that $\Gamma$ satisfies the (CAF) condition. Given a probability measure $v$ on $\Lambda_X\Gamma$, the \textbf{$L$-barycenter} $B_v(L)$ of $v$ is defined by the subset of $X$ satisfying:
 \[
\setdef{x\in X}{\int_{\xi\in\partial X}\beta_\xi(x,p)dv\leq\inf_{y\in X}\int_{\xi\in\partial X}\beta_\xi(y,p)dv+L}
\]

\begin{lemma}\label{L: Barycenter}
If $v$ has no atoms, the set $B_v(L)$ is bounded for all $L>0$.
\end{lemma}
\begin{proof}
    Let $F_v: X\to \bbR$ be the function given by
\[F_v(x)=\int_{\xi\in\partial X}\beta_\xi(x,p) dv.\]
It suffices to show that $F_v(x)\to\infty$ as $d(x,p)\to \infty.$  

Rewrite \[
F_v(x)=\int_{\xi\in\partial X}\left(d(x,p)-2(\xi,x)_p\right)dv=d(x,p)-2\int_{\xi\in\partial X}(\xi,x)_pdv.
\] 
For each $x\in X$, let
\[
E(x):=\setdef{\xi\in \Lambda_X\Gamma}{(\xi,x)_p\geq \frac{d(x,p)}{3}}.
\]
Then
\[
2\int_{\xi\in \Lambda_X\Gamma}(x,\xi)_pdv\leq 2d(x,p)v(E(x))+\frac{2d(x,p)}{3}(1-v(E(x))).
\]
If $\xi\in E(x)$, then by hyperbolicity,
\[
E(x)\subset \setdef{\zeta\in \Lambda_X\Gamma}{(\zeta,\xi)_p\geq\frac{d(x,p)}{3}-\delta}.
\]
Since $\Lambda_X\Gamma$ is compact and $v$ has no atoms, there exists a $K>0$ such that $v(E(x))\leq \frac{1}{8}$ when $d(x,p)\geq K$. It follows that for all $x$ with $d(x,p)\geq K$, $F_v(x)\geq \frac{1}{6}d(x,p)$. 

This completes the proof.
\end{proof}
\begin{lemma}
$\Gamma$ satisfies the (CAF) condition.
\end{lemma}
\begin{proof}
Let $N$ be an amenable commensurated subgroup of $\Gamma$. Then there exists an $N$-invariant probability measure $v$ on the compact space $\Lambda_X\Gamma$.

 \emph{Case 1:} The support of $v$ contains a point with finite $N$-orbit. Let
\[
F_N=\{\xi\in\Lambda_X\Gamma:\ |N\xi|<\infty\}.
\]
Since $N$ is commensurated, the property of having finite orbit is unchanged after replacing $N$ by a commensurable subgroup. Consequently, $F_N$ is $\Gamma$-invariant.

Choose $\xi\in F_N$. Replacing $N$ by a finite-index subgroup if necessary, we may assume that $N$ fixes $\xi$. Since the action of $\Gamma$ on $\Lambda_X\Gamma$ is minimal and $\Gamma$ is non-elementary, we can choose $\gamma,\eta\in\Gamma$ such that
$\xi,\gamma\xi,\eta\xi$
are pairwise distinct.

Since $N$ is commensurated, the subgroup
\[
N'
=
N\cap\gamma N\gamma^{-1}\cap\eta N\eta^{-1}
\]
has finite index in $N$. Moreover, $N'$ fixes $\xi$, $\gamma\xi$, and $\eta\xi$. Thus $N'$ fixes at least three distinct points of $\Lambda_X\Gamma$.

By the geometric mean-value property, the metric derivatives of elements of $N'$ are uniformly bounded above and below. As in the proof of Lemma~\ref{L: finite}, it follows that $N'$ is finite. Since $N'$ has finite index in $N$, the group $N$ is finite.

\emph{Case 2:} The support of $v$ contains no point with $N$-finite orbit. Then $v$ has no atoms, and by Lemma~\ref{L: Barycenter}, $B_v(1)$ is bounded. Since $N$ preserves the bounded subset $B_v(1)$, $N$ must be finite. 

This establishes the (CAF) condition of $\Gamma$.
\end{proof}
By the following theorem, the amenable radical of $H$ is compact.
\begin{theorem}[\cite{bader2020lattice}, Theorem 5.2]\label{T: semisimple}
 The amenable radical of a lattice envelope of a group with property (CAF) is compact.   
\end{theorem}
\begin{remark}    
Although the definition of the lattice envelope in \cite{bader2020lattice} is assumed to be compactly generated, the proof of Theorem~\ref{T: semisimple} is independent of this condition.
\end{remark}
Indeed, the amenable radical of $H$ is trivial.
\begin{lemma}
   $H$ has trivial compact radical. 
\end{lemma}
\begin{proof}
    Assume that $K$ is a compact normal subgroup of $H$. There exists a constant $\kappa$ such that for all $f\in K$ and $x\in \Lambda_X\Gamma$, $|\ln f'(x)|<\kappa$. 

    Let $\gamma\in \Gamma$ be hyperbolic. Then, for every $f\in K$:
    $$|\ln (\gamma^{n}f\gamma^{-n})'(\gamma^-)|<\kappa.$$
    Since
    $$\ln (\gamma^{n}f\gamma^{-n})'(\gamma^-)=\ln (\gamma^{n})'(f(\gamma^-))+\ln f'(\gamma^-)+\ln (\gamma^{-n})'(\gamma^-).$$ 
    Thus
    $|\ln (\gamma^n)'(f(\gamma^-))+n\epsilon\ell_\gamma|<2\kappa$ for all $n$. Taking a fixed $n>\lfloor\frac{2\kappa}{\epsilon\ell_\gamma}\rfloor$, we have $d_\epsilon(\gamma^+,K\gamma^-)>0$.

    On the other hand, $K\gamma^-=\gamma K\gamma^{-1}\gamma^-=\gamma K\gamma^-$. Hence $K\gamma^-=\gamma^-$. Since this holds for all hyperbolic elements, $K=\{Id\}$. Therefore, $H$ has a trivial compact radical.
\end{proof}

By Theorem~\ref{T: Stucture}, $H$ is virtually a product. The next lemma states that one of the two factors is trivial.
\begin{lemma}\label{L: simple}
The group $H$ is not virtually a direct product of two nontrivial groups.
\end{lemma}

\begin{proof}
Suppose that $H$ contains an open finite-index subgroup
\[
H_0=H_1\times H_2,
\]
where both $H_1$ and $H_2$ are nontrivial. Let
\[
p_i:H_0\longrightarrow H_i,
\qquad i=1,2,
\]
be the coordinate projections, and regard each $H_i$ as a subgroup of
$H_0$.

Set
\[
\Gamma_0=\psi(\partial\Gamma)\cap H_0.
\]
Since $H_0$ has finite index in $H$, the subgroup $\Gamma_0$ has finite index in $\psi(\partial\Gamma)$. Hence
\[
\Lambda_X\Gamma_0=\Lambda_X\Gamma.
\]
As the action of $\psi(\partial\Gamma)$ on $\Lambda_X\Gamma$ is minimal, the same is true for $\Gamma_0$. In particular, the action of $H_0$ on $\Lambda_X\Gamma$ is minimal.

Choose a hyperbolic element
\[
\gamma\in\Gamma_0,
\]
and write
\[
\gamma=p_1(\gamma)p_2(\gamma).
\]
Since $p_i(\gamma)$ commutes with $\gamma$, it preserves the set
\[
\{\gamma^+,\gamma^-\}.
\]
In fact, it fixes both points individually. Indeed, the attracting and repelling fixed points of $\gamma$ are dynamically characterized by the north--south dynamics of $\gamma$, and any homeomorphism commuting with $\gamma$ preserves this dynamical characterization.

At the attracting fixed point $\gamma^+$, the chain rule gives
\[
\gamma'(\gamma^+)
=
p_1(\gamma)'(\gamma^+)\,
p_2(\gamma)'(\gamma^+).
\]
Since
\[
\gamma'(\gamma^+)<1,
\]
after exchanging the two factors if necessary, we may assume that
\[
p_1(\gamma)'(\gamma^+)<1.
\]

We claim that $\gamma^+$ is an isolated fixed point of $p_1(\gamma)$. Suppose otherwise that there exists a sequence of distinct fixed points
\[
\xi_n\longrightarrow\gamma^+.
\]
By the geometric mean-value property,
\[
p_1(\gamma)'(\gamma^+)\,
p_1(\gamma)'(\xi_n)=1.
\]
Hence
\[
p_1(\gamma)'(\xi_n)
=
\frac{1}{p_1(\gamma)'(\gamma^+)}
>1
\]
for every $n$. On the other hand, continuity of the metric derivative gives
\[
p_1(\gamma)'(\xi_n)
\longrightarrow
p_1(\gamma)'(\gamma^+)<1,
\]
a contradiction. Thus $\gamma^+$ is an isolated fixed point of $p_1(\gamma)$.

We next claim that the fixed-point set of $p_1(\gamma)$ is 
\[
\operatorname{Fix}\bigl(p_1(\gamma)\bigr)
=
\{\gamma^+,\gamma^-\}.
\]
Indeed, suppose
\[
\xi\in
\operatorname{Fix}\bigl(p_1(\gamma)\bigr)
\setminus\{\gamma^+,\gamma^-\}.
\]
Since $p_1(\gamma)$ commutes with $\gamma$, every point $\gamma^n\xi$ is fixed by $p_1(\gamma)$. By the north--south dynamics of $\gamma$,
\[
\gamma^n\xi\longrightarrow\gamma^+
\]
as $n\to+\infty$, contradicting the fact that $\gamma^+$ is an isolated fixed point.

Since every element of $H_2$ commutes with $p_1(\gamma)$, it preserves
\[
\operatorname{Fix}\bigl(p_1(\gamma)\bigr)
=
\{\gamma^+,\gamma^-\}.
\]
Moreover, the same north--south argument shows that each element of $H_2$ fixes $\gamma^+$ and $\gamma^-$ individually. Therefore,
\[
\{\gamma^+,\gamma^-\}
\subseteq
\operatorname{Fix}(H_2).
\]

The set $\operatorname{Fix}(H_2)$ is closed and $H_0$-invariant. It is invariant under $H_2$ by definition, and under $H_1$ because $H_1$ commutes with $H_2$. Since the $H_0$-action on $\Lambda_X\Gamma$ is minimal, we conclude that
\[
\operatorname{Fix}(H_2)=\Lambda_X\Gamma.
\]
The action of $H$ on $\Lambda_X\Gamma$ is faithful, so
\[
H_2=\{1\},
\]
a contradiction.
\end{proof}

Finally, we can classify the possible structures of $H$.
\begin{theorem}\label{T: Structure 1}
Let $\Gamma$ act metrically properly and non-elementarily on a good hyperbolic space $(X,d)$. Assume that $\Gamma$ is a $2\frac{1}{2}$-lattice and that its limit set $\Lambda_X\Gamma$ is compact. Then $\psi(\partial\Gamma)$ is a lattice in the group $H$ constructed
in Section~\ref{lattice}, and one of the following holds:
\begin{enumerate}
\item\label{case 1} $H$ is discrete.
\item\label{case 2} $H$ is a nondiscrete, totally disconnected locally compact
group with trivial amenable radical.
\item\label{case 3} There exists a connected, center-free, simple real Lie group
$G^0$ of real rank one such that
\[
G^0\leq H\leq\operatorname{Aut}(G^0).
\]
In particular, $H$ is either $G^0$ or $\operatorname{Aut}(G^0)$.
\end{enumerate}
\end{theorem}

\begin{proof}
Suppose that $H$ does not fall into cases~(\ref{case 1}) or~(\ref{case 2}). By
Theorem~\ref{T: Stucture}, the identity component
\[
G^0=H^\circ
\]
is a connected, center-free, semisimple real Lie group without compact factors, and it has finite index in $H$.

By Lemma~\ref{L: simple}, $G^0$ cannot decompose as a direct product of
two nontrivial groups. Hence $G^0$ is simple.

Let
\[
\Gamma^0=\psi(\partial\Gamma)\cap G^0.
\]
Since $G^0$ is open and of finite index in $H$, the group $\Gamma^0$ is a lattice in $G^0$ and has finite index in $\psi(\partial\Gamma)$. 

If $G^0$ had real rank at least two, then by Haettel's theorem \cite{haettel2020hyperbolic} the action of the lattice $\Gamma^0$ on $X$ would be elementary. Since this conclusion is obtained from boundary dynamics, it is unchanged after passing to finite-index subgroups or finite extensions. As $\Gamma^0$ has finite index in $\psi(\partial\Gamma)$, and the kernel of $\Gamma\to\psi(\partial\Gamma)$ is finite, it would follow that the action of $\Gamma$ on $X$ is elementary, contradicting the assumption. Therefore, $G^0$ has real rank one.

Since $G^0$ is a characteristic subgroup of $H$, conjugation defines a continuous homomorphism
\[
c:H\longrightarrow\operatorname{Aut}(G^0).
\]
Its kernel is
\[
\Ker(c)=C_H(G^0).
\]
Because $G^0$ is center-free,
\[
C_H(G^0)\cap G^0={1}.
\]
Since $H/G^0$ is finite, it follows that $C_H(G^0)$ is finite. Moreover, $C_H(G^0)$ is a closed normal subgroup of $H$. Since $H$ has trivial
compact radical, we conclude that
\[
C_H(G^0)={1}.
\]
Thus, $c$ is injective.

The restriction of $c$ to $G^0$ identifies $G^0$ with
$\operatorname{Inn}(G^0)$. We may therefore identify
\[
G^0\leq H\leq\operatorname{Aut}(G^0).
\]
Equivalently,
\[
H/G^0\leq\operatorname{Out}(G^0).
\]
For a connected, center-free, simple real Lie group of real rank one,
the group $\operatorname{Out}(G^0)$ is either trivial or isomorphic to
$\mathbb{Z}/2\mathbb{Z}$. Consequently,
\[
H=G^0
\qquad\text{or}\qquad
H=\operatorname{Aut}(G^0).
\]
\end{proof}
\begin{corollary}\label{C: Structure intermediate}
Let $L<H$ be a closed subgroup containing $\partial\Gamma$. Then $\partial\Gamma$ is a lattice in $L$, and one of the following holds:
\begin{enumerate}
    \item $L$ is discrete;
    \item $L$ is a nondiscrete, totally disconnected locally compact group with trivial amenable radical;
    \item there exists a connected, center-free, simple real Lie group $G^0$ of real rank one such that
    \[
    G^0\leq L\leq \operatorname{Aut}(G^0).
    \]
\end{enumerate}
\end{corollary}

\begin{proof}
By Corollary~\ref{C: intermediate lattice}, $\partial\Gamma$ is a lattice in $L$. Since $\Gamma$ satisfies the \emph{(CAF)} condition, Theorem~\ref{T: semisimple} implies that the amenable radical of $L$ is compact.

The proof that $H$ has trivial compact radical applies verbatim to $L$. Indeed, if $K$ is a compact normal subgroup of $L$, then $K$ is normalized by $\partial\Gamma$, and the argument using hyperbolic elements of $\Gamma$ shows that $K$ fixes the fixed points of every hyperbolic element. Since such fixed points are dense in $\Lambda_X\Gamma$, it follows that $K$ acts trivially on $\Lambda_X\Gamma$. As the action of $L$ is faithful, $K$ is trivial. Thus $L$ has trivial amenable radical.

We may therefore apply Theorem~\ref{T: Stucture} to $L$. The proof of Lemma~\ref{L: simple} also applies verbatim to $L$, since $L$ contains $\partial\Gamma$ and every finite-index subgroup of $\partial\Gamma$ acts minimally on $\Lambda_X\Gamma$. Hence $L$ is not virtually a direct product of two nontrivial groups.

It follows that either $L$ is totally disconnected or its identity component $G^0=L^\circ$ is a connected, center-free, simple real Lie group without compact factors. In the latter case, $G^0$ has finite index in $L$. As in the proof of Theorem~\ref{T: Structure 1}, Haettel's theorem rules out real rank at least two, and hence $G^0$ has real rank one.

Finally, conjugation gives an injective homomorphism
\[
L\longrightarrow \operatorname{Aut}(G^0),
\]
since the centralizer $C_L(G^0)$ is finite and normal, hence trivial by the triviality of the compact radical. Therefore,
\[
G^0\leq L\leq \operatorname{Aut}(G^0),
\]
which proves the result.
\end{proof}

\section{Proofs of Theorem~\ref{T: arithmetic} and \ref{T: locally}}\label{sec: proofs}
The proofs of Theorems~\ref{T: arithmetic} and \ref{T: locally} rely on the following lemma.

 Every $h\in\GCommen{\Gamma}{d}$ induces a rough isometry on $(\Gamma,d)$, and thus a boundary map $\partial h\in \Homeo(\Lambda_X\Gamma)$.
\begin{lemma}\label{L: comm}
The kernel of the map
\[
\partial:\GCommen{\Gamma}{d}\longrightarrow \Homeo(\Lambda_X\Gamma)
\]
is trivial, and
\[
\partial(\GCommen{\Gamma}{d})\cap H
\subset
\Commen_H(\partial\Gamma).
\]
Furthermore, if
\[
[\GCommen{\Gamma}{d}:\Gamma]=\infty,
\]
then
\[
[\partial(\GCommen{\Gamma}{d})\cap H:\partial\Gamma]=\infty.
\]
\end{lemma}

\begin{proof}
Without loss of generality, we replace $X$ with $(\Gamma,d)$.

We first prove that $\Ker(\partial)$ is trivial. Let
\[
s:\Gamma_1\longrightarrow \Gamma_2
\]
represent an element of $\GCommen{\Gamma}{d}$ such that $\partial s=\operatorname{Id}$.

Recall that we identify $(\Gamma,d)$ with the orbit $\Gamma p$. Choosing coset representatives for $\Gamma_1$ and $\Gamma_2$, we may extend $s$ to a rough self-isometry of $(\Gamma,d)$, which we continue to denote by $s$.

Let $\nu$ be a nonatomic probability measure on $\Lambda_X\Gamma$. Since $\partial s=\operatorname{Id}$, we have
\[
s_*\nu=\nu.
\]
If $s$ is a $(1,C)$-rough isometry, then
\[
s(B_\nu(1))\subset B_\nu(1+2C).
\]
The same argument applies to $\gamma_*\nu$. Since
\[
B_{\gamma_*\nu}(A)=\gamma B_\nu(A),
\]
Lemma~\ref{L: Barycenter} gives a constant $M>0$, independent of $\gamma$, such that
\[
d(s(\gamma),\gamma)\leq M
\]
for every $\gamma\in\Gamma_1$.

For $\gamma\in\Gamma_1$, write
\[
s(\gamma)=\gamma\tau(\gamma),
\]
where
\[
d(\tau(\gamma),e)\leq M.
\]
Since $s$ is a homomorphism, for all $\gamma_1,\gamma_2\in\Gamma_1$ we have
\[
\tau(\gamma_1\gamma_2)
=
\gamma_2^{-1}\tau(\gamma_1)\gamma_2\tau(\gamma_2).
\]
Consequently,
\[
d(\gamma_2^{-1}\tau(\gamma_1)\gamma_2,e)
\leq 2M
\]
for all $\gamma_1,\gamma_2\in\Gamma_1$, and hence
\[
d(\tau(\gamma_1)\gamma_2,\gamma_2)\leq 2M.
\]
Since $\Gamma_1$ has finite index in $\Gamma$, it is at finite Hausdorff distance from $\Gamma$. Thus $\tau(\gamma_1)$ acts at bounded distance from the identity on $(\Gamma,d)$, and therefore induces the identity on $\Lambda_X\Gamma$. Hence
\[
\tau(\gamma_1)\in
N:=\Ker(\Gamma\longrightarrow\partial\Gamma)
\]
for every $\gamma_1\in\Gamma_1$.

By Lemma~\ref{L: finite in H}, the group $N$ is finite. Let
\[
\Gamma_1^0
=
\ker\bigl(\Gamma_1\longrightarrow\Aut(N)\bigr).
\]
For $\gamma_1,\gamma_2\in\Gamma_1^0$, the preceding identity becomes
\[
\tau(\gamma_1\gamma_2)
=
\tau(\gamma_1)\tau(\gamma_2).
\]
Thus
\[
\tau|_{\Gamma_1^0}:\Gamma_1^0\longrightarrow N
\]
is a homomorphism. Since $N$ is finite,
\[
\Gamma_1^1:=\ker(\tau|_{\Gamma_1^0})
\]
has finite index in $\Gamma_1$. On $\Gamma_1^1$ we have
\[
s=\operatorname{Id}.
\]
Therefore $s$ represents the identity element of $\ACommen(\Gamma)$, proving that
\[
\Ker(\partial)=\{1\}.
\]

We next prove the commensurator inclusion. Let
\[
s:\Gamma_1\longrightarrow\Gamma_2
\]
represent an element of $\GCommen{\Gamma}{d}$. For every $\gamma\in\Gamma_1$,
\[
\partial s\,\partial\gamma\,(\partial s)^{-1}
=
\partial\bigl(s\gamma s^{-1}\bigr).
\]
Hence
\[
\partial s\,\partial\Gamma_1\,(\partial s)^{-1}
=
\partial\Gamma_2.
\]
Since $\Gamma_1$ and $\Gamma_2$ have finite index in $\Gamma$, both $\partial\Gamma_1$ and $\partial\Gamma_2$ have finite index in $\partial\Gamma$. Therefore, whenever $\partial s\in H$,
\[
\partial s\in\Commen_H(\partial\Gamma).
\]
Thus
\[
\partial(\GCommen{\Gamma}{d})\cap H
\subset
\Commen_H(\partial\Gamma).
\]

It remains to prove the final assertion. Let
\[
s:\Gamma_1\longrightarrow \Gamma_2
\]
represent an element of $\GCommen{\Gamma}{d}$.  We normalize the Bowen-Margulis-Sullivan measure of a finite-index subgroup $\Gamma'<\Gamma$ by 
\[
\BMSm{\Gamma'}(\Gamma'\backslash\Lambda_X\Gamma^{(2)}\times\mathbb{R})=1.
\]
Then 
\[
\BMSm{\Gamma_1}
=
\frac{1}{[\Gamma:\Gamma_1]}\BMSm{\Gamma},
\]
and
\[
\BMSm{\Gamma_2}
=
\frac{1}{[\Gamma:\Gamma_2]}\BMSm{\Gamma}.
\]
Moreover,
\[
s_*\BMSm{\Gamma_1}
=
\BMSm{\Gamma_2}.
\]
Therefore,
\[
s_*\BMSm{\Gamma}
=
\frac{[\Gamma:\Gamma_1]}
{[\Gamma:\Gamma_2]}
\BMSm{\Gamma}.
\]

Define
\[
\Delta:\GCommen{\Gamma}{d}\longrightarrow \mathbb{Q}_{>0}^{\times}
\]
by
\[
\Delta(s)
=
\frac{[\Gamma:\Gamma_1]}
{[\Gamma:\Gamma_2]}.
\]
This is independent of the choice of representative of $s$, and defines a group homomorphism. By the definition of $H$,
\[
\partial s\in H
\quad\Longleftrightarrow\quad
\Delta(s)=1.
\]
Hence
\[
\partial(\GCommen{\Gamma}{d})\cap H
=
\partial(\Ker\Delta).
\]
Also,
\[
\Gamma<\Ker(\Delta).
\]

Assume now that
\[
[\GCommen{\Gamma}{d}:\Gamma]=\infty.
\]
If
\[
[\Ker(\Delta):\Gamma]=\infty,
\]
then, since $\Ker(\partial)$ is trivial,
\[
[\partial(\Ker\Delta):\partial\Gamma]=\infty.
\]
Thus,
\[
[\partial(\GCommen{\Gamma}{d})\cap H:\partial\Gamma]
=
\infty,
\]
and we are done.

Suppose instead that
\[
[\Ker(\Delta):\Gamma]<\infty.
\]
Set
\[
K=\Ker(\Delta).
\]
Since $K$ is a normal subgroup of $\GCommen{\Gamma}{d}$, for every representative
\[
s:\Gamma_1\longrightarrow\Gamma_2
\]
we have
\[
s\Gamma_1s^{-1}=\Gamma_2, \quad sKs^{-1}=K.
\]
Therefore conjugation by $s$ gives an isomorphism between $\Gamma_1$ and $\Gamma_2$ inside $K$, and hence
\[
[K:\Gamma_1]=[K:\Gamma_2].
\]
Since
\[
\Gamma_1,\Gamma_2<\Gamma<K,
\]
we have
\[
[K:\Gamma_1]
=
[K:\Gamma][\Gamma:\Gamma_1]
\]
and
\[
[K:\Gamma_2]
=
[K:\Gamma][\Gamma:\Gamma_2].
\]
It follows that
\[
[\Gamma:\Gamma_1]
=
[\Gamma:\Gamma_2].
\]
Therefore
\[
\Delta(s)=1.
\]
Since $s$ was arbitrary,
\[
\GCommen{\Gamma}{d}
=
\Ker(\Delta).
\]
Consequently,
\[
[\GCommen{\Gamma}{d}:\Gamma]
=
[\Ker(\Delta):\Gamma]
<
\infty,
\]
contradicting the assumption.

Hence
\[
[\Ker(\Delta):\Gamma]=\infty.
\]
Since $\Ker(\partial)$ is trivial, we conclude that
\[
[\partial(\GCommen{\Gamma}{d})\cap H:\partial\Gamma]
=
\infty.
\]

This completes the proof.
\end{proof}
\subsection{Proof of Theorem~\ref{T: arithmetic}}

Let
\[
N=\Ker\bigl(\Gamma\longrightarrow \partial\Gamma\bigr).
\]
By Lemma~\ref{L: finite in H}, the group $N$ is finite. We identify
\[
\bar{\Gamma}=\Gamma/N
\]
with the boundary image $\partial\Gamma$.

Let
\[
C_{\mathrm{com}}=\partial\bigl(\GCommen{\Gamma}{d}\bigr)\cap H
\]
and define
\[
H_{\mathrm{com}}=\overline{C_{\mathrm{com}}}^{H}.
\]
By Lemma~\ref{L: comm},
\[
C_{\mathrm{com}}\subseteq\Commen_H(\partial\Gamma),
\]
and
\[
[C_{\mathrm{com}}:\partial\Gamma]=\infty.
\]
Moreover, $C_{\mathrm{com}}$ is dense in $H_{\mathrm{com}}$ by definition.

The group $H_{\mathrm{com}}$ is a closed subgroup of $H$ containing
$\partial\Gamma$. By Corollary~\ref{C: intermediate lattice}, the group $\partial\Gamma$ is a
lattice in $H_{\mathrm{com}}$.

We first observe that $H_{\mathrm{com}}$ is nondiscrete. Indeed, if it
were discrete, then the density of $C_{\mathrm{com}}$ would imply
\[
C_{\mathrm{com}}=H_{\mathrm{com}}.
\]
Since a lattice in a discrete group has finite index, this would give
\[
[C_{\mathrm{com}}:\partial\Gamma]<\infty,
\]
contrary to Lemma~\ref{L: comm}.

The structural classification in Corollary~\ref{C: Structure intermediate}, applied
to the lattice envelope $H_{\mathrm{com}}$, now gives two possibilities.

Suppose first that $H_{\mathrm{com}}$ is a nondiscrete, totally
disconnected locally compact group. Then $\partial\Gamma$ is a lattice in
$H_{\mathrm{com}}$, and
\[
C_{\mathrm{com}}
\subseteq
\Commen_{H_{\mathrm{com}}}(\partial\Gamma)
\]
is dense in $H_{\mathrm{com}}$. This is the first conclusion of the
theorem.

Suppose next that there is a connected, center-free, simple real Lie
group $G^0$ of real rank one such that
\[
G^0
\leq
H_{\mathrm{com}}
\leq
\operatorname{Aut}(G^0).
\]
Set
\[
\Gamma^0=\partial\Gamma\cap G^0.
\]
Then $\Gamma^0$ is a lattice in $G^0$ and has finite index in
$\partial\Gamma$. Since $C_{\mathrm{com}}$ is dense in $H_{\mathrm{com}}$ and
commensurates $\partial\Gamma$, the commensurator of $\Gamma^0$ is dense in
$G^0$. By the commensurator criterion, $\Gamma^0$, and hence
$\partial\Gamma$, is arithmetic.

It remains to identify the metric. Let $\mathbf H$ be the rank-one symmetric space associated with $G^0$. Since
\[
H_{\mathrm{com}}\subset H,
\]
the metric derivative defines a continuous cocycle
\[
c:H_{\mathrm{com}}\times\Lambda_X\Gamma\longrightarrow\mathbb R,
\qquad
c(h,\xi)=\frac{1}{\epsilon}\log h'(\xi),
\]
which restricts to the Busemann cocycle associated with $(X,d)$ on $\Lambda_X\Gamma$.

For a rank-one simple real Lie group, the space of such cocycles, up to strict equivalence, is one-dimensional \cite{zimmer2013ergodic}*{Proposition 4.2.13}. Therefore, after multiplication by a positive constant, the cocycle $c$ is strictly equivalent to the standard Busemann cocycle of $\mathbf H$. Consequently, there exists $\lambda>0$ such that
\[
\lambda\ell_d(\gamma)
=
\ell_{\mathbf H}(\gamma)
\]
for every infinite-order element $\gamma$ in a finite-index subgroup of $\partial\Gamma$.

Both actions are non-elementary and hence contain contracting elements. The marked-length-spectrum rigidity theorem for actions with contracting elements \cite{wan2025marked}*{Theorem 5.1} therefore applies and gives a rough isometry
\[
(\partial\Gamma,\lambda d)
\longrightarrow
\partial\Gamma\cdot o\subset\mathbf H
\]
for some $o\in\mathbf H$.

Thus the induced metric on
\[
\bar\Gamma=\Gamma/N
\]
is arithmetic.
\qed

\subsection{Proof of Theorem~\ref{T: locally}}
Let
\[
N=\Ker\bigl(\Gamma\longrightarrow \partial\Gamma\bigr).
\]
As above, $N$ is finite, and we identify
\[
\bar{\Gamma}=\Gamma/N
\]
with $\partial\Gamma$.

Define
\[
H_{\mathrm{rough}}=\overline{\mathrm{H}_\Gamma(X)}^{H}.
\]
By definition, this is a closed subgroup of $H$ containing $\partial\Gamma$. Hence, by Corollary~\ref{C: intermediate lattice}, $\partial\Gamma$ is a lattice in $H_{\mathrm{rough}}$.

The group $H_{\mathrm{rough}}$ is nondiscrete. Indeed, if it were
discrete, the density of $\mathrm{H}_\Gamma(X)$ in
$H_{\mathrm{rough}}$ would imply
\[
\mathrm{H}_\Gamma(X)=H_{\mathrm{rough}}.
\]
Since $\partial\Gamma$ is a lattice in the discrete group
$H_{\mathrm{rough}}$, it would have finite index. This would contradict
the assumption
\[
[\mathrm{H}_\Gamma(X):\partial\Gamma]=\infty.
\]

Applying Corollary~\ref{C: Structure intermediate} to $H_{\mathrm{rough}}$, we again
obtain two possibilities.

If $H_{\mathrm{rough}}$ is a nondiscrete, totally disconnected locally
compact group, then $\partial\Gamma$ is a lattice in
$H_{\mathrm{rough}}$. This gives the first conclusion of the theorem.

Otherwise, there is a connected, center-free, simple real Lie group
$G^0$ of real rank one such that
\[
G^0
\leq
H_{\mathrm{rough}}
\leq
\operatorname{Aut}(G^0).
\]
Let $\mathbf H$ be the corresponding rank-one symmetric space. Then
$\partial\Gamma\cap G^0$ is a lattice in $G^0$, and a finite-index subgroup of $\partial\Gamma$ is a lattice in
the full isometry group of $\mathbf H$.

As in the proof of Theorem~\ref{T: arithmetic}, the Busemann cocycle of
$(X,d)$ extends to $H_{\mathrm{rough}}$. Since the relevant space of
Busemann cocycles for a rank-one simple Lie group is one-dimensional up
to strict equivalence \cite{zimmer2013ergodic}*{Proposition 4.2.13}, there exists $\lambda>0$ such that the marked
length spectrum of $(\partial\Gamma,\lambda d)$ agrees with the marked length
spectrum of the orbit metric in $\mathbf H$. The coarse
marked-length-spectrum rigidity theorem  \cite{wan2025marked}*{Theorem 5.1} then gives a rough isometry
\[
(\partial\Gamma,\lambda d)
\longrightarrow
\partial\Gamma\cdot o\subseteq\mathbf H
\]
for some $o\in\mathbf H$.

Therefore, the induced metric on
\[
\bar{\Gamma}=\Gamma/N
\]
is rank-one locally symmetric. This proves
Theorem~\ref{T: locally}.
\qed

\end{document}